\documentclass[11pt,a4paper]{article}
\usepackage{amsfonts,amsmath,amssymb,amsthm,amsxtra,color,graphicx,latexsym,mathtools,xcolor}
\usepackage{geometry}

\newtheorem{definition}{Definition}
\newtheorem{proposition}[definition]{Proposition}
\newtheorem*{theorem}{Theorem}
\newtheorem{remark}[definition]{Remark}

\numberwithin{equation}{section}
\numberwithin{definition}{section}

\newcommand{\Iff}{if\textcompwordmark f}
\newcommand{\rn}{\ensuremath{\left(\varPhi,\vect{y}\right)}}
\newcommand{\rns}{\ensuremath{\left(\varPhi^*\!,\vect{y}\right)}}

\newcommand{\vect}[1]{\boldsymbol{#1}}

\DeclareMathOperator{\RANK}{rank}
\DeclareMathOperator{\DIFF}{d\!}

\DeclareMathOperator{\TR}{tr}

\newcommand{\TU}[1]{\textup{#1}}

\begin{document}
\title{Bonnet's type theorems in the relative differential geometry of the 4-dimensional space}

\author{{Stylianos Stamatakis and Ioannis Kaffas}\\ \emph{Aristotle University of Thessaloniki}\\ \emph{Department of Mathematics}\\ \emph{GR-54124 Thessaloniki, Greece}\\  \emph{e-mail: stamata@math.auth.gr}}
\date{}
\maketitle

\renewcommand{\thefootnote}{}



\renewcommand{\thefootnote}{arabic{footnote}}
\setcounter{footnote}{0}

\begin{abstract}
\noindent We deal with hypersurfaces in the framework of the relative differential geometry in $\mathbb{R}^4$.
We consider a hypersurface $\varPhi$ in $\mathbb{R}^4$ with position vector field $\vect{x}$ which is relatively normalized by a relative normalization $\vect{y}$.
Then $\vect{y}$ is also a relative normalization of every member of the one-parameter family $\mathcal{F}$ of hypersurfaces $\varPhi_\mu$ with position vector field $\vect{x}_\mu = \vect{x} + \mu \, \vect{y}$, where $\mu$ is a real constant.
We call every hypersurface $\varPhi_\mu \in \mathcal{F}$ relatively parallel to $\varPhi$.
This consideration includes both Euclidean and Blaschke hypersurfaces of the affine differential geometry.
In this paper we express the relative mean curvature's functions of a hypersurface $\varPhi_\mu$ relatively parallel to $\varPhi$ by means of the ones of $\varPhi$ and the ``relative distance" $\mu$.
Then we prove several Bonnet's type theorems. More precisely, we show that if two relative mean curvature's functions of $\varPhi$ are constant, then there exists at least one relatively parallel hypersurface with a constant relative mean curvature's function.

\end{abstract}

\section{Introduction}\label{Section1}
            Starting point for this paper is the following Theorem of Ossian Bonnet (see \cite{Bonnet}, \cite{Kreyszig}):
\begin{theorem}
            Let $\varPhi$ be a surface in the 3-dimensional Euclidean space $\mathbb{R}^3$.\\
            $(a)$ If $\varPhi$ has a constant mean curvature $H = 1/(2\mu)$, its  parallel surfaces at the distances $2\mu$ and $\mu$ have  the constant mean curvature $-1/(2\mu)$ and the  constant Gaussian curvature $1/\mu^2$, respectively.\\
            $(b)$ If $\varPhi$ has a constant Gaussian curvature $K = 1/\mu^2$, its parallel surface at a distance $\pm \mu$ has a constant mean curvature $\mp 1/(2\mu)$.
\end{theorem}
            In a previous paper (see \cite{Stamatakis6}) the authors and I. Delivos
            extended this Theorem to the relative differential geometry in $\mathbb{R}^3$ in the following way:

            Let $\varPhi$ be a surface in the space $\mathbb{R}^3$ with position vector field $\vect{x}$, which is relatively normalized by a relative normalization $\vect{y}$.
            Consider the one-parameter family of surfaces $\varPhi_\mu$ with position vector field
\begin{equation}\label{01}
            \vect{x}_\mu = \vect{x} + \mu \, \vect{y},
\end{equation}
            where $\mu$ is a nonvanishing real constant.
            Then $\vect{y}$ is also a relative normalization of $\varPhi_\mu$.
            For an obvious reason, every surface $\varPhi_\mu$ of this family is called a relatively parallel surface of $\varPhi$ at the ``relative distance" $\mu$.
            Then we proved that the same results, which states the above Theorem of O. Bonnet, hold true in the relative differential geometry in $\mathbb{R}^3$ if the Gaussian curvature and the mean curvature are replaced by the relative curvature and the relative mean curvature, respectively.

            In this paper we work in the framework of the relative differential geometry in $\mathbb{R}^4$.
            Our approach contains the Euclidean case and the case of Blaschke hypersurfaces in affine differential geometry.
            We consider a hypersurface $\varPhi$ of $\mathbb{R}^4$ with position vector field $\vect{x}$, which is relatively normalized by a relative normalization $\vect{y}$.
            Then $\vect{y}$ is also a relative normalization of every member of the one-parameter family $\mathcal{F}$ of hypersurfaces $\varPhi_\mu$ with position vector field of the form \eqref{01}, where $\mu$ is a real nonvanishing constant.
            We call every hypersurface $\varPhi_\mu \in \mathcal{F}$ relatively parallel to $\varPhi$.

            As a first step we express the relative mean curvature's functions of a  hypersurface $\varPhi_\mu$ relatively parallel to $\varPhi$ by means of the ones of $\varPhi$ and the ``relative distance" $\mu$.

            Then we prove several Bonnet's type theorems, namely we show that if two relative mean curvature's functions of $\varPhi$ are constant, then there exists at least one relatively parallel hypersurface with a constant relative mean curvature's function.

\section{Preliminaries}\label{Section2}
           This section contains some basic definitions, formulae and results on relative differential geometry; for this purpose we have used the book \cite{Schirokow} and the monograph  \cite{Simon1991} as  general references.

            We consider a $C^{r}$-hypersurface $\varPhi = (M,\vect{x})$ in $\mathbb{R}^{n+1} $
            defined by an $n$-dimensional, oriented, connected $C^{r}$-manifold $M$, $r \geq 3$, and by a $C^{r}$-immersion  $\vect{x} \colon M \rightarrow \mathbb{R}^{n+1}$,   whose Gaussian curvature $\widetilde{K}$ never vanishes on $M$. Let $\vect{\xi}$ be the unit normal vector field to $\varPhi$ and
\begin{align}
            II &\coloneqq -  \langle \DIFF \vect{x}, \DIFF \vect{\xi} \rangle \eqqcolon h_{ij}\DIFF u^i \DIFF u^j, \quad i,j = 1,\dotsc,n,                     \label{2.11}
\end{align}
            be 
            the second 
            fundamental form of $\varPhi$,
            where $\langle \, , \, \rangle$ denotes the standard scalar product in $\mathbb{R}^{n+1}$ and $(u^1,u^2, \dotsc, u^n) \in M$ are local coordinates.

            We denote by $\partial_i f$, $\partial_j\partial_i f$ etc. the partial derivatives of a $C^r$-function (or a vector-valued function) $f$ with respect to $u^i$.
            A  $C^{r}$-mapping $\vect{y} \colon M \rightarrow \mathbb{R}^{n+1}$ is called a $C^{r}$-relative normalization of $\varPhi$, if
\begin{subequations}\label{2.15}
\begin{align}
            \RANK \left(\Big\{\partial_1 \vect{x}, \dotsc, \partial_n \vect{x}, \vect{y}\Big\}\right) &= n + 1,\\
            \RANK \left( \Big \{ \partial_1 \vect{x}, \dotsc, \partial_n \vect{x}, \partial_i \vect{y} \Big\} \right) &= n \quad \forall \,\, i = 1, \dotsc, n,
\end{align}
\end{subequations}
            The pair $(\varPhi,\vect{y})$ is called a relatively normalized hypersurface in $\mathbb{R} ^{n+1}$ and the straight line issuing from a point $P \in \varPhi$ in the direction of $\vect{y}$ is called the relative normal of \rn{} at $P$. The pair $\overline{\varPhi}=(M,\vect{y})$ is called the relative image of $(\varPhi,\vect{y})$.

            The covector $\vect{X}$ of the tangent vector space is defined by
\begin{equation}\label{2.20}
            \langle \vect{X},\partial_i\vect{x}\rangle = 0 \quad
             \text{and}\quad \langle \vect{X},\vect{y}\rangle = 1.
\end{equation}
            The quadratic  form
\begin{equation}\label{2.25}
            G \coloneqq - \langle \DIFF \vect{x}, \DIFF \vect{X}  \rangle
\end{equation}
            is called the relative metric of \rn{}. 
            For its coefficients $G_{ij}$ the following relations hold
\begin{equation*}
            G_{ij} = - \langle  \partial_i \vect{x}, \partial_j \vect{X} \rangle = \langle  \partial_j \partial_i \vect{x}, \vect{X} \rangle.
\end{equation*}
             The support function of Minkowski of the relative normalization $\vect{y}$ is defined by
\begin{equation}\label{2.35}
            q \coloneqq \langle \vect{\xi},\vect{y}\rangle \colon M \rightarrow \mathbb{R} ,\quad q\in C^{r-1}(M),
\end{equation}
            and, by virtue of \eqref{2.15}, never vanishes on $M$.

            Conversely, the relative normalization is determined by means of the support function through
\begin{equation}\label{2.40}
            \vect{y}=\nabla^{III}\!\!\left( q,\, \vect{x} \right) + q \,\vect{\xi},
\end{equation}
            where $\nabla^{III}$ denotes the first Beltrami-operator with respect to the third fundamental form $III$ of $\varPhi$ (see~ \cite[p.~197]{fM89}, \cite{Stamatakis6}).

            Because of \eqref{2.20}
\begin{equation}\label{2.45}
            \vect{X} = q^{-1} \vect{\xi}, \quad  G_{ij}=q^{-1}h_{ij},  \quad G^{(ij)} = q\, h^{(ij)},
\end{equation}
            where $h^{(ij)}$ and $G^{(ij)}$ are the inverse of the tensors $h_{ij}$ and  $G_{ij}$, respectively. From now on we shall use $G_{ij}$ for ``raising and lowering" the indices in the sense of the classical tensor notation.

            We consider the bilinear form
\begin{equation}\label{2.65}
            B \coloneqq \langle \DIFF \vect{y}, \DIFF \vect{X}  \rangle.
\end{equation}
            For its coefficients $B_{ij}$ we have
\begin{equation}\label{2.70}
            B_{ij} = \langle  \partial_i \vect{y}, \partial_j \vect{X}\rangle =  - \langle  \partial_j \partial_i \vect{y}, \vect{X}\rangle.
\end{equation}
            Then the following Weingarten type equations  are valid
\begin{equation}\label{2.80}
            \partial_i \vect{y} = -B_{i}^{j} \,  \partial_j  \vect{x}.
\end{equation}
            Let $T_P \varPhi$ be the tangent vector space of $\varPhi$ at a point $P \in \varPhi$. By means of \eqref{2.80} the relative shape (or Weingarten) operator
\begin{equation*}
            \omega \colon T_P \varPhi \rightarrow T_P \varPhi
\end{equation*}
            of the relatively normalized hypersurface \rn{} at $P$ is defined such that
\begin{equation*}
            B(\vect{u},\vect{v}) = G(\omega( \vect{u}),\vect{v})
\end{equation*}
            for $\vect{u},\vect{v} \in T_P \varPhi$ (see \cite[p.~66]{Simon1991}), or equivalently such that
\begin{equation*}
            \omega(\partial_i \vect{x}) = - \partial_i \vect{y}.
\end{equation*}
            Special mention should be made of the fact that the relative differential geometry includes both the Euclidean one, which arises for $q = 1$, or equivalently for $\vect{y} = \vect{\xi}$, and the equiaffine one, which is based upon the equiaffine normalization $\vect{y}_{\TU{\tiny{AFF}}}$.
            The last normalization is defined, on account of \eqref{2.40}, by means of the equiaffine support function
\begin{equation*}
            q_{\TU{\tiny{AFF}}} \coloneqq |\widetilde{K}|^\frac {1} {n+2}.
\end{equation*}
            The real eigenvalues of the relative shape operator $\omega$ of \rn{} are called relative principal curvatures of \rn{} and denoted by $\kappa_i$, $i = 1, \dotsc, n$  (see \cite{Simon1991}).
            Their reciprocals $R_i$ (when $\kappa_i \neq 0$) are called relative radii of curvature.

            In what follows we shall consider only relatively normalized hypersurfaces such that their relative shape operator has $n$ real eigenvalues $\kappa_i, i = 1,\dotsc,n$ (not necessarily  different).
            Their normed elementary symmetric functions are called relative mean curvature's functions of \rn{} and denoted as follows:
\begin{align}
            H_1 &\coloneqq  \frac{1}{n} \left(\kappa_1 +  \dotsb + \kappa_n \right),        \label{2.90}\\
            H_2 &\coloneqq  \frac{1}{\binom{n}{2}} \left(\kappa_1 \, \kappa_2 + \dotsb + \kappa_{n-1} \, \kappa_n \right),                                                           \label{2.95}\\
            \vdots              \nonumber           \\
            H_n &\coloneqq  \kappa_1 \,\kappa_2 \dots \kappa_n. \label{2.100}
\end{align}
            In particular, the first relative mean curvature's function $H_1$ is denoted by $H$ and called  relative mean curvature and the $n-$th  relative mean curvature's function $H_n$ is denoted by $K$ and called  relative (Gauss-Kronecker) curvature of \rn{}.

\section{Relatively parallel hypersurfaces in $\mathbb{R}^{n+1}$}\label{Section3}
            Let \rn{} be a relatively normalized hypersurface in the  space $\mathbb{R}^{n+1}$.
            In what follows we suppose that the relative normalization $\vect{y} = \vect{y}(u^i)$ is a $C^r$-immersion.

            We consider the one-parameter family of mappings $\vect{x}_\mu \colon M \rightarrow \mathbb{R}^{n+1}$ which are defined by
\begin{equation}\label{3.10}
            \vect{x}_\mu(u^i) = \vect{x}(u^i) + \mu \, \vect{y}(u^i),
            \end{equation}
            where $\mu$ is a real nonvanishing constant.
            From \eqref{2.80} and \eqref{3.10} we obtain
\begin{equation}\label{3.15}
            \partial_i \vect{x}_\mu  = \left(\delta^j_i - \mu \, B^j_i\right)\, \partial_j \vect{x},
\end{equation}
            where $\delta_i^j$ is the Kronecker delta.
            Then it is readily verified that the vector product of the partial derivatives $\partial_i \vect{x}_\mu $ satisfies the relation
\begin{equation*}
            \partial_1 \vect{x}_\mu \times \dotsb \times \partial_n  \vect{x}_\mu = A(\mu) \;
            \Big( \partial_1 \vect{x}   \times \dotsb \times \partial_n  \vect{x}\Big),
\end{equation*}
            where
\begin{equation}\label{3.25}
            A(\mu)\coloneqq  \det \!\left(\delta_i^j -\mu \, B_i^j \right).
\end{equation}
            We suppose throughout that $A(\mu) \neq 0$ everywhere on $M$.
            Then the one-parameter family \eqref{3.10} consists of $C^r$-immersions.

            In this way we obtain the one-parameter family
\begin{equation*}
            \mathcal{F} \eqqcolon \left \{ \, \varPhi_\mu \eqqcolon (M,\vect{x}_\mu) \mid \vect{x}_\mu = \vect{x} + \mu \, \vect{y},\, \mu \in \mathbb{R} \setminus\{0\} \,  \right  \}
\end{equation*}
            of $C^r$-hypersurfaces.
            To the point $P(u^i_0)$ of $\varPhi$ corresponds the point $P_\mu(u^i_0)$ of $\varPhi_\mu$ so that their position vectors are $\vect{x}(u^i_0)$ and  $\vect{x}_\mu(u^i_0)$, respectively.

            From \eqref{3.15} we infer that the tangent hyperplanes to each member of the family $\mathcal{F}$ and to $\varPhi$ at corresponding points are parallel; $\varPhi$ and every $\varPhi_\mu \in \mathcal{F}$ are in Peterson correspondence \cite{Chakmazyan}.

            Furthermore, as we can see immediately by using \eqref{3.10}, the relations \eqref{2.15} are valid as well if the parametrization $\vect{x}(u^i)$ of $\varPhi$ is replaced by the parametrization $\vect{x}_\mu(u^i)$ of $\varPhi_\mu \in \mathcal{F}$.
            Therefore $\vect{y}$ is a relative normalization for each member of $\mathcal{F}$.
            We call each relatively normalized hypersurface $\left( \varPhi_\mu ,\vect{y} \right)$ 
            a relatively parallel hypersurface to \rn.
            Throughout what follows, we shall freely use for $\mu$  the expression ``relative distance".

            By means of \eqref{2.35}, (\ref{2.45}a) and  \eqref{2.70} it is clear that \rn{} \emph{and every relatively parallel hypersurface $\left( \varPhi_\mu ,\vect{y} \right)$ to \rn{} have in common}
            (a) \emph{the relative image} $\overline{\varPhi}$,\\
            (b) \emph{the support function} $q$,\\
            (c) \emph{the covector $\vect{X}$ of their tangent vector spaces and}\\
            (d) \emph{the quadratic differential form \eqref{2.65}}.

            For simplicity we denote a relatively parallel hypersurface $\left( \varPhi_\mu ,\vect{y} \right)$ at a relative distance $\mu$ by \rns.
            Analogously, we mark all the corresponding quantities induced by \rns{} with  an asterisk and we refer by (\#*) to the formula, which, on this modification, flows from formula (\#).

            The second fundamental form $II^*$ of \rns{} can be calculated by using \eqref{2.11}, (\ref{2.11}*), (\ref{2.45}a), \eqref{2.65}, \eqref{3.10} and $\langle \DIFF \vect{y}, \vect{X}\rangle = 0.$ We find
\begin{equation*}
            II^* = II - \mu \,q\, B.
\end{equation*}
            By combining \eqref{2.25}, (\ref{2.25}*) and \eqref{3.15} we calculate the relative metric $G^*$ of \rns{}
\begin{equation*}
                G^* = G - \mu \, B.
\end{equation*}
            In addition to relation \eqref{3.15} we have likewise from \eqref{3.10}
\begin{equation}\label{4.25}
            \partial_i \vect{x}  = \left(\delta^j_i + \mu \, B^{*j}_{\phantom{^*}i} \right) \, \partial_j \vect{x}^*.
\end{equation}
            Then from \eqref{2.80}, (\ref{2.80}*), \eqref{3.15} and \eqref{4.25} we obtain
\begin{equation}\label{4.30}
            B^j_i = B_{\phantom{^*}i}^{*k}\,(\delta^j_k - \mu \, B^j_k),
            \qquad B^{*j}_{\phantom{^*}i} = B_i^{k} \, (\delta^j_k + \mu \, B^{*j}_{\phantom{^*}k})
\end{equation}
and
\begin{equation*}
            B^\kappa_i\, B_{\phantom{^*}\kappa}^{*j} = B^{*\kappa}_{\phantom{^*}i} \, B_\kappa^j.
\end{equation*}
In the following sections we shall discuss relatively parallel hypersurfaces in the  space $\mathbb{R}^4$.

\section{Relatively parallel hypersurfaces in $\mathbb{R}^4$}\label{Section5}
                The relative principal curvatures $\kappa_1,\kappa_2$ and $\kappa_3$ of a relatively normalized hypersurface \rn{} in $\mathbb{R}^4$, are the roots of the characteristic polynomial
\begin{equation*}
                P_\omega(\kappa) = \det \! \left( B_i^j - \kappa \,\delta_i^j \right), \quad i,j = 1,2,3,
\end{equation*}
            of the shape operator $\omega$, or, what is the same, the roots of the equation
\begin{equation*}
            \det \! \left( B_{ij} - \kappa \,G_{ij} \right) = 0.
\end{equation*}
            Consequently we have
\begin{align}
            H    &= \frac{1}{3} \TR \! \left( B_i^j \right),           \label{5.35} \\
            H_2  &= \frac{1}{3} \left(B_1^1 \, B_2^2 + B_2^2\, B_3^3 + B_3^3 \, B_1^1 - B_1^2\, B_2^1 - B_2^3 \, B_3^2 - B_1^3 \,B_3^1   \right),       \label{5.40}\\
            K    &= \det \! \left( B_i^j \right)                       \label{5.45}.
\end{align}
            We consider a relatively parallel hypersurface \rns{} to \rn{} at a relative distance $\mu$.
            We recall the function $A(\mu)$ which is defined by \eqref{3.25}. In our case ($n = 3$) we have
\begin{equation}\label{5.50}
                A(\mu)  =  -\mu^3 \,K + 3 \mu^2 \, H_2  - 3 \mu \, H  + 1.
\end{equation}
            We notice that when the relative curvature $K$ of \rn{} does not vanish, $A(\mu)$ can be written by means of the relative principal radii of curvature $R_1,R_2,R_3$ of \rn{} as follows
\begin{equation}\label{5.55}
                A(\mu) = -K \left(\mu - R_1 \right) \left(\mu - R_2 \right) (\mu - R_3).
\end{equation}
            Hence, in case $K \neq 0$, since $A(\mu) \neq 0$, we have $\mu \neq R_1,R_2,R_3$.
            We solve now the system (\ref{4.30}a) (or (\ref{4.30}b))  with respect to the mixed components of the shape operator $\omega^*$ of \rns{} and we find
            \begin{align*}
            B_{\phantom{^*}1}^{*1} &= \frac{B_1^1 - \mu \left(B_1^1B_2^2 + B_1^1B_3^3 - B_1^2B_2^1 - B_1^3B_3^1\right) + \mu^2 K}{A(\mu)}, \\
            B_{\phantom{^*}1}^{*2} &= \frac{B_1^2 + \mu \left(B_1^3B_3^2 - B_1^2B_3^3\right)}{A(\mu)}, \\
            B_{\phantom{^*}1}^{*3} &= \frac{B_1^3 + \mu \left(B_1^2B_2^3 - B_1^3B_2^2\right)}{A(\mu)}, \\
            B_{\phantom{^*}2}^{*1} &= \frac{B_2^1 + \mu \left(B_2^3B_3^1 - B_2^1B_3^3\right)}{A(\mu)}, \\
            B_{\phantom{^*}2}^{*2} &= \frac{B_2^2 - \mu \left(B_1^1B_2^2 + B_2^2B_3^3 - B_2^3B_3^2 - B_1^2B_2^1\right) + \mu^2 K}{A(\mu)}, \\
            B_{\phantom{^*}2}^{*3} &= \frac{B_2^3 + \mu \left(B_2^1B_1^3 - B_2^3B_1^1\right)}{A(\mu)}, \\
            B_{\phantom{^*}3}^{*1} &= \frac{B_3^1 + \mu \left(B_3^2B_2^1 - B_3^1B_2^2\right)}{A(\mu)},\\
            B_{\phantom{^*}3}^{*2} &= \frac{B_3^2 + \mu \left(B_3^1B_1^2 - B_3^2B_1^1\right)}{A(\mu)}, \\
            B_{\phantom{^*}3}^{*3} &= \frac{B_3^3 - \mu \left(B_1^1B_3^3 + B_2^2B_3^3 - B_1^3B_3^1 - B_2^3B_3^2\right) + \mu^2 K}{A(\mu)}.
            \end{align*}
            From (\ref{5.35}*)--(\ref{5.45}*) we obtain the relative mean curvature's functions of \rns{}:
\begin{align}
            K^* & = \frac{K} {A(\mu)},                                   \label{5.60}\\
            H^*_{2} & = \frac{ - \mu \, K + H_2 } {A(\mu)},              \label{5.65}\\
            H^* & =  \frac{\mu^2 \, K  - 2 \mu \, H_2  + H} {A(\mu)}.    \label{5.70}
\end{align}
            Then, taking into account (\ref{2.90})--(\ref{2.100}), (\ref{2.90}*)--(\ref{2.100}*) and   \eqref{5.60}--\eqref{5.70}, we find
\begin{equation}\label{5.75}
            \kappa_i^* = \frac{\kappa_i}{1 - \mu \, \kappa_i}, \quad i = 1,2,3.
\end{equation}
            Hence for $\kappa_i \neq 1/\mu$ we have $\kappa_i^* = 0$ \Iff{} $\kappa_i = 0$ and for the relative principal radii of curvature $R_i^*$ of \rns{} we have
\begin{equation}\label{5.76}
            R_i^* = R_i - \mu.
\end{equation}
            The relations \eqref{5.75} and \eqref{5.76} agree with formulae (n) of \cite[p.~117]{Opozda}, albeit obtained with a quite different way.

            In the remainder of this section we suppose that the relative curvature $K$ of the relatively hypersurface \rn{} does not vanish.

            On account of \eqref{2.90}--\eqref{2.100}, for $n = 3$ we obtain
\begin{equation}\label{5.15}
            R_1 + R_2 + R_3 = \frac{3H_2}{K}.
\end{equation}

            Formula \eqref{5.76} combined with (\ref{5.15}*) allows us to prove the following proposition:
\begin{proposition}
            Let \rn{} be a relatively normalized hypersurface in the  space $\mathbb{R}^4$ with nonvanishing constant sum of its relative principal radii of curvature and nonvanishing relative curvature $K$. Then every relatively parallel hypersurface to \rn{} has also constant sum of its relative radii of curvature and there is exactly one relatively parallel hypersurface to \rn{} with vanishing second relative mean curvature.
\end{proposition}
\begin{proof}
            The first part of the proposition follows from \eqref{5.76}. Put
\begin{equation*}
            R_1 + R_2 + R_3 = 3\,c = const., \quad c \neq 0.
\end{equation*}
            For the relatively parallel hypersurface \rns{} at the relative distance $\mu = c$ we have
\begin{equation*}
            R^*_1 + R^*_2 + R^*_3 = 0,
\end{equation*}
            which on account of (\ref{5.15}*) yields $H^*_2 = 0$. The uniqueness of $\mu$ follows from \eqref{5.65}.
\end{proof}
\begin{remark}
            When $R_1 + R_2 + R_3 = 0$, or equivalently $H_2 = 0$, it follows from \eqref{5.60} and \eqref{5.65} 
            $$
            H^*_2 / K^* = -\mu.
            $$
            Hence the sum of the relative radii of curvature of every relatively parallel hypersurface to \rn{} is constant, but there is not any relatively parallel hypersurface with vanishing second relative mean curvature.
\end{remark}
            Moreover, we find from \eqref{5.70}:

            (a) \emph{When} $H_2 = 0$, \emph{there are at most two relatively parallel to} \rn{} \emph{hypersurfaces which are relatively minimal.}

            (b)  \emph{When}
\begin{equation*}
            H_2^2 - K H \, \geq 0 \quad \text{\emph{and}} \quad \frac{H_2 +(-1)^i \sqrt{H_2^2 - K  H}}{K} = c_i = const. \neq 0, \quad i = 1,2,
\end{equation*}
            \emph{then the relatively parallel to} \rn{} \emph{hypersurface at the relative distance}
            $\mu = c_1$ (\emph{or} $\mu = c_2$) \emph{is relatively minimal}.

            (c) \emph{When the conditions}
            $$H_2 / K = c_3 \neq 0 \quad \text{\emph{and}} \quad H_2^2 = K H$$
            \emph{are valid, then the unique relatively parallel to} \rn{} \emph{hypersurface at the relative distance} $\mu = c_3$, \emph{which has vanishing second relative mean curvature, is relatively minimal.}

            By using \eqref{5.60}--\eqref{5.70} we obtain
\begin{equation*}
            H^{*2}_{2} -  K^*  H^* = \frac{H^{2}_{2} -  K  H}{A(\mu)^2},
\end{equation*}
            which 
            may be rewritten as
\begin{equation*}
            \frac {H^{*2}_{2} -  K^* H^*}{ K^{*2}} =  \frac{H_2^2 - K H}{K^2}
\end{equation*}
            and we arrive at the following
\begin{proposition}
            Let \rn{} be a relatively normalized hypersurface in the  space $\mathbb{R}^4$ with nonvanishing relative curvature $K$. Then the function
\begin{equation*}
            \frac{H_2^2 - K H}{K^2}
\end{equation*}
            remains invariant by the transition to anyone of the relatively parallel hypersurfaces to \rn.
\end{proposition}
            Then, from \eqref{5.60} and \eqref{5.65}, we take
\begin{equation}\label{5.95}
            \mu = \frac{H_2}{K} - \frac{H^*_2}{K^*},
\end{equation}
i.e. \emph{the function}
\begin{equation*}
            \frac{H_2}{K} - \frac{H^*_2}{K^*}
\end{equation*}
            \emph{is independent of the point} $P(u^i) \in \varPhi$ \emph{for every relatively parallel hypersurface to} \rn {}.

            Substituting $\mu$ from \eqref{5.95} in \eqref{5.70} we obtain
\begin{equation*}
\begin{split}
            K K^*  H_2^* \bigg[K^2  H_2^* + 3 K^* H^* \Big(H_2^2 - K H \Big) \bigg] - K^3 H_2^{*3}  H^*  = \\
            K^{*3} \bigg[ 2H_2^3  H^* + K  H_2 \Big(H_2 - 3H  H^* \Big) + K^2 \Big( H^* - H\Big)\bigg].
\end{split}
\end{equation*}
\section{Bonnet's type theorems}\label{Section6}
            We can now formulate the theorems to which this article is dedicated, namely the Bonnet's type theorems for relatively parallel hypersurfaces in the  space $\mathbb{R}^4$.
\begin{proposition}\label{Prop61}
            Let \rn{} be a relatively normalized hypersurface in the  space $\mathbb{R}^4$ with  constant  relative curvature and constant second relative mean curvature such that
\begin{equation*}
            K \neq 0 \quad \text{and} \quad             K^2 \geq H^{3}_{2}.
\end{equation*}
            Then there is a relatively parallel hypersurface to \rn{} which has constant relative mean curvature.
\end{proposition}
\begin{proof}
            By virtue of \eqref{5.50} and \eqref{5.70} we have
\begin{equation}\label{6.10}
            A(\mu) \big(1 + 3\mu \, H^* \big) = P(\mu),
\end{equation}
where $$P(\mu) = 2 \mu^3 K - 3 \mu^2 \, H_2 + 1.$$
            The only real root of the polynomial $P(\mu)$ is
\begin{equation}\label{6.20}
            \mu_1 = -\frac{1}{2K}\left[W + H_2 \left(\frac{H_2}{W} -1  \right) \right],
\end{equation}
            where
\begin{equation*}
             W = \left[ 2K^2 -H_2^3 + 2\left| K \right| \sqrt{K^2 - H_2^3} \right]^{1/3}.
\end{equation*}
            By a straightforward calculation we get $A(\mu_1) \neq 0$. From \eqref{6.10} we see that $1 + 3\mu_1 \, H^*$ must vanish. Hence
\begin{equation*}
            H^* = -\frac{1}{3\mu_1} = const. \qedhere
\end{equation*}
\end{proof}
\begin{remark}
            In the special case where $$K^2 = H^{3}_{2},$$ we have from \eqref{6.20} $$\mu_1 = -1 / \left( 2K^{1/3}\right),$$ whereupon  $$H^* = 2K^{1/3}/3.$$
\end{remark}
\begin{proposition}\label{Prop63}
            Let \rn{} be a relatively normalized hypersurface in the  space $\mathbb{R}^4$ with  constant relative curvature and constant relative mean curvature.\\
            (a) If
\begin{equation*}
            K \neq 0 \quad \text{and} \quad K(K - 2H^3) \geq 0,
\end{equation*}
            then there is a relatively parallel hypersurface to \rn{} which has constant second relative mean curvature.\\
            (b) If
\begin{equation*}
            K \neq 0 \quad \text{and} \quad K(K - H^3) \geq 0,
\end{equation*}
            then there is a relatively parallel hypersurface to \rn{} which has constant relative mean curvature.
\end{proposition}
\begin{proof}
            (a) From \eqref{5.50} and \eqref{5.65} we have
\begin{equation}\label{6.30}
            A(\mu) \big(1 - 3\mu^2 \, H_2^*\big) = P(\mu),
\end{equation}
where $$P(\mu) = 2 \mu^3 K - 3 \mu \, H + 1$$.
            The only real root of the polynomial $P(\mu)$  is
\begin{equation}\label{6.40}
            \mu_2 = -\frac{H}{\sqrt[3]{2}\,\, W} - \frac{W}{\sqrt[3]{4}\,\,K},
\end{equation}
            where
\begin{equation*}
            W = \left[ K^2 + \sqrt{K^3 \left(K - 2H^3 \right)} \right]^{1/3}.
\end{equation*}
            Because of $A(\mu_2) \neq 0$, we have $$1 - 3\mu_2^2 \, H_2^* = 0$$ (see \eqref{6.30}). Hence
\begin{equation*}
            H_2^* = \frac{1}{3\mu_2^2} = const.
\end{equation*}
            (b) From \eqref{5.50} and \eqref{5.70} we have
\begin{equation}\label{6.50}
            A(\mu) \big(2 + 3\mu \, H^* \big) = P(\mu),
\end{equation}
where $$P(\mu) = \mu^3 K - 3 \mu \, H + 2$$.
            The only real root of the polynomial $P(\mu)$  is
\begin{equation}\label{6.60}
            \mu_3 = -\frac{H}{W} - \frac{W}{K},
\end{equation}
            where
\begin{equation*}
             W = \left[ K^2 + \sqrt{K^3 \left(K - H^3 \right)} \right]^{1/3}.
\end{equation*}
            By virtue of $A(\mu_3) \neq 0$, from \eqref{6.50} we take $$2 + 3\mu_3 \, H^* = 0.$$ Hence
\begin{equation*}
            H^* = -\frac{2}{3\mu_3} = const.             \qedhere
\end{equation*}
\end{proof}
\begin{remark}
            (a) In the special case where $$K = 2H^3,$$ we have from \eqref{6.40} $$\mu_2 = -1 / H,$$ whereupon  we find $$H_2^* = H^2 / 3.$$\\
            (b) In the special case where $$K = H^3,$$  we have from \eqref{6.60} $$\mu_3 = -2 / H,$$ whereupon $$H^* = H / 3$$. Beyond this,  
            we can see that the parallel hypersurface to \rn{} at the relative distance $$\mu_3 = 1 / H$$ possesses the constant mean curvature $$H^* = -2H / 3.$$
\end{remark}
\begin{proposition}\label{Prop65}
            Let \rn{} be a relatively normalized hypersurface in the  space $\mathbb{R}^4$ with  constant second relative mean curvature and constant relative mean curvature.\\
            (a) If
\begin{equation*}
            H_2 \neq 0 \quad \text{and} \quad 3H^2 \geq 4 H_2,
\end{equation*}
            then there are two relatively parallel hypersurfaces to \rn{} which have constant relative curvature.\\
            (b) If
\begin{equation*}
            H_2 \neq 0 \quad \text{and} \quad 9H^2 \geq 8H_2,
\end{equation*}
            then there are two relatively parallel hypersurfaces to \rn{} which have constant second relative mean curvature.\\
            (c) If
\begin{equation*}
            H_2 \neq 0 \quad \text{and} \quad H^2 \geq H_2,
\end{equation*}
            then there are two relatively parallel hypersurfaces to \rn{} which have constant relative  mean curvature.
\end{proposition}
\begin{proof}
            From \eqref{5.55} and \eqref{5.60}--\eqref{5.70} we have
\begin{align}
            A(\mu) \big(1 + \mu^3  K^* \big)   & = P_1(\mu),        \label{6.70}\\
            A(\mu) \big(1 - \mu^2  H_2^* \big) & = P_2(\mu),  \label{6.80}\\
            A(\mu) \big(1 + \mu \, H^* \big)   & = P_3(\mu), ,     \label{6.90}
            \end{align}
where
\begin{align*}
P_1(\mu) &= 3\mu^2 H_2 - 3 \mu \, H + 1, \\
  P_2(\mu)& = 2\mu^2 H_2 - 3 \mu \, H + 1,    \\
    P_3(\mu)& = \mu^2 H_2 - 2 \mu \, H + 1.
\end{align*}
 For the roots
\begin{alignat}{2}
            \mu_i &= \frac{3H +(-1)^i \sqrt{3} \sqrt{3H^2 - 4H_2}}{6H_2}, \quad && i =4,5, \label{6.100} \\
            \mu_i &= \frac{3H +(-1)^i \sqrt{9H^2 - 8H_2}}{4H_2}, \quad && i =6,7, \label{6.110} \\
            \mu_i &= \frac{H +(-1)^i \sqrt{H^2 - H_2}}{H_2}, \quad && i =8,9, \label{6.120}
\end{alignat}
            of the polynomials $P_1(\mu),P_2(\mu),P_3(\mu)$, respectively, we have $A(\mu_i) \neq 0$, $i = 4, \dots 9$. Hence, by means of \eqref{6.70}--\eqref{6.90}
\begin{alignat*}{2}
            K^*   &= -\frac{1}{\mu_i^3} = const., \quad && i = 4,5, \\
            H_2^* &=  \phantom{-}\frac{1}{\mu_i^2} = const., \quad && i = 6,7 \\
            H^*   &= -\frac{1}{\mu_i}   = const., \quad && i = 8,9.  \qedhere
\end{alignat*}
\end{proof}
\begin{remark}
            (a) If $$3H^2 = 4H_2$$ we have from \eqref{6.100} $$\mu_4 = \mu_5 = 2/(3H).$$ In this case there is one relatively parallel hypersurface to \rn{} which has constant relative curvature $$K^* = -27H^3/8.$$\\
            (b) If $$9H^2 = 8H_2$$ we have from \eqref{6.110} $$\mu_6 = \mu_7  = 2/(3H).$$ In this case there is one relatively parallel hypersurface to \rn{} which has constant second relative  mean curvature $$H_2^* = 9H^2/4.$$\\
            (c) If $$H^2 = H_2$$ we have from \eqref{6.120} $$\mu_8 = \mu_9  = 1/H.$$ In this case there is one relatively parallel hypersurface to \rn{} which has constant relative mean curvature $$H^* = -H.$$
\end{remark}

\end{document}